\documentclass[11pt, a4paper]{article}
\usepackage[margin=1in]{geometry}
\usepackage{graphicx}
\usepackage{amsmath}
\usepackage{amsthm}
\usepackage{amsfonts}
\usepackage[utf8]{inputenc}
\usepackage{xcolor}
\usepackage{amssymb}
\usepackage[hidelinks]{hyperref}
\usepackage{fullpage}
\usepackage{esvect,mathtools}
\usepackage[numbers,square,sort]{natbib}
\usepackage[capitalize]{cleveref}
\usepackage{xurl}
\hypersetup{breaklinks=true}

\newtheorem{theorem}{Theorem}
\numberwithin{theorem}{section} 
\newtheorem{lemma}[theorem]{Lemma}
\newtheorem{corollary}[theorem]{Corollary}
\newtheorem{question}[theorem]{Question}
\newtheorem{conjecture}[theorem]{Conjecture}
\theoremstyle{remark}
\newtheorem{remark}[theorem]{Remark}

\theoremstyle{definition}
\newtheorem{definition}[theorem]{Definition}

\newcommand\ab[1]{\lvert #1\rvert}

\title{Oriented Ramsey numbers of graded digraphs}
\author{Patryk Morawski\thanks{Department of Computer Science, ETH Z\"urich, Z\"urich, Switzerland.
			Email: \href{mailto:pmorawski@student.ethz.ch}{\nolinkurl{pmorawski@student.ethz.ch}}. }
   \and
   Yuval Wigderson%
   \thanks{Institute for Theoretical Studies, ETH Z\"urich, Z\"urich, Switzerland. Email: \href{mailto:yuval.wigderson@eth-its.ethz.ch}{\nolinkurl{yuval.wigderson@eth-its.ethz.ch}}.
   Supported by Dr.\ Max R\"{o}ssler, the Walter Haefner Foundation, and the ETH Z\"{u}rich Foundation.}
   }
\date{}

\begin{document}

\maketitle
\begin{abstract}
    We show that any graded digraph $D$ on $n$ vertices with maximum degree $\Delta$ has an oriented Ramsey number of at most $C^\Delta n$ for some absolute constant $C > 1$, improving upon a recent result of Fox, He, and Wigderson. In particular, this implies that oriented grids in any fixed dimension have linear oriented Ramsey numbers, and gives a polynomial bound on the oriented Ramsey number of the hypercube.
    
    We also show that this result is essentially best possible, in that there exist graded digraphs on $n$ vertices with maximum degree $\Delta$ such that their oriented Ramsey number is at least $c^\Delta n$ for some absolute constant $c > 1$.
\end{abstract}
\section{Introduction}\label{section:introduction}

\subsection{Background}
The Ramsey number $r(G)$ of a graph $G$ is the minimum number $N$ such that every 2-coloring of the edges of the complete graph $K_N$ on $N$ vertices contains a monochromatic copy of $G$.
Studying how $r(G)$ depends on $G$ is the main question in graph Ramsey theory, and is one of the most studied topics in combinatorics. In particular, one would like to obtain comparable upper and lower bounds on $r(G)$ in terms of basic parameters of $G$.
Such bounds are known in certain regimes; for example, it is known that if $G$ has $n$ vertices, then $r(G)$ grows exponentially in $n$ if and only if $G$ has $\Omega(n^2)$ edges \cite{conlon2013ramsey,sudakov2011conjecture}.

In particular, sparse graphs---those with $o(n^2)$ edges---have subexponential Ramsey numbers. However, it was observed very early in the history of graph Ramsey theory (see e.g.\ \cite{erdos1975partition}) that certain sparse graphs, such as trees and cycles, have \emph{much} smaller Ramsey numbers, namely \emph{linear} in their order.
In 1975, Burr and Erd\H{o}s \cite{burr_magnitude_1975} conjectured that this is in fact true for all sparse graphs.
Their notion of sparsity is the degeneracy\footnote{In fact, they originally stated their conjecture in terms of the \emph{arboricity} of $G$, but it is easy to verify that a graph has bounded degeneracy if and only if it has bounded arboricity. Most subsequent papers on this topic use degeneracy, which seems more well-suited for Ramsey-theoretic study.} of the graph, where a graph $G$ is called $d$-degenerate if every subgraph of $G$ has a vertex of degree at most $d$, and the degeneracy of $G$ is the least $d$ for which it is $d$-degenerate.
The Burr--Erd\H{o}s conjecture then states that for every $d$-degenerate graph $G$ we have that $r(G) \leq C_d |V(G)|$, where $C_d$ is a constant that depends only on $d$.
The first major step towards proving this conjecture was made by Chv\'{a}tal, R\"odl, Szemer\'{e}di and Trotter \cite{chvatal1983ramsey}, who proved that that for every graph $G$ with maximum degree $\Delta$, we have $r(G) \leq C_\Delta |V(G)|$, that is, that the conjecture holds under the stronger assumption of bounded maximum degree.
After a sequence of further partial results (e.g.\ \cite{fox2009two, kostochka2003ramsey, kostochka2001graphs, chen1993graphs, graham2000graphs}), the full Burr--Erd\H{o}s conjecture was finally resolved by Lee \cite{lee2017ramsey} in 2017.

For bounded-degree graphs, we have a fairly precise understanding of how large their Ramsey numbers are. Substantially improving on the early result of Chvat\'al, R\"odl, Szemer\'edi, and Trotter \cite{chvatal1983ramsey}, Graham, R\"odl, and Ruci\'nski \cite{graham2000graphs} proved that every $n$-vertex graph $G$ with maximum degree $\Delta$ satisfies $r(G) \leq C^{\Delta (\log \Delta)^2}n$, where $C$ is an absolute constant. In a subsequent paper \cite{graham2001bipartite}, they noted that their technique yields a better bound of $r(G) \leq C^{\Delta \log \Delta}$ in case $G$ is bipartite, and also showed that this bound is close to best possible, in that there exist bipartite $n$-vertex graphs maximum degree $\Delta$ satisfying $r(G) \geq c^{\Delta} n$, for another absolute constant $c>1$. Subsequent work by Conlon \cite{conlon2009hypergraph}, Fox--Sudakov \cite{fox2009density}, and Conlon--Fox--Sudakov \cite{conlon2012two} removed one logarithmic factor from both upper bounds; in particular, it is now known that every $n$-vertex bipartite graph with maximum degree $\Delta$ satisfies $r(G) \leq C^{\Delta}n$, which is best possible up to the value of $C$.

We now turn our attention to directed graphs (\emph{digraphs} for short), where similar questions can be asked, and which are the main focus of this paper.
For an acyclic\footnote{A digraph is \emph{acyclic} if it contains no oriented cycles.} digraph $D$, we define the \emph{oriented Ramsey number} $\vv{r}(D)$ as the minimum number $N$ such that every tournament on $N$ vertices, that is, every edge-orientation of the complete graph $K_N$, contains a copy of $D$.
The study of oriented Ramsey numbers was initiated in 1951 by Stearns \cite{stearns1959voting}, who showed that for a transitive tournament $\vv{T_n}$ on $n$ vertices we have $\vv{r}(\vv{T_n}) \leq 2^{n-1}$, which was complemented by a lower bound of $\vv{r}(\vv{T_n}) \geq 2^{\frac{n}{2} - 1}$ by Erd\H{o}s and Moser \cite{erdos1964representation} in 1964.
As in the undirected setting, for sparser digraphs $D$ this number is in general much smaller than exponential in $|V(D)|$; for example, for $n > 8$ every orientation of the $n$-vertex path has oriented Ramsey number equal to $n$ \cite{thomason1986paths, havet2000oriented}.
For more general oriented trees, Sumner conjectured that $\vv{r}(T) \leq 2n - 2$ for any oriented tree $T$ on $n$ vertices.
Sumner's conjecture has attracted a great deal of interest over the years (e.g.\ \cite{thomason1986paths, dross2021unavoidability, el2004trees, haggkvist1991trees, havet2002trees, havet2000median, benford2022few, benford2022many, kuhn2011approximate}) and in 2011 it was proved for sufficiently large $n$ by K\"uhn, Mycroft, and Osthus \cite{kuhn2011proof}.

Motivated by these results, Buci\'c, Letzter and Sudakov \cite{bucic2019directed} asked whether a natural anologue of the Burr--Erd\H{o}s conjecture holds for acyclic digraphs, that is, whether for all acyclic digraphs $D$ with maximum degree\footnote{By the maximum degree of a digraph, we mean the maximum degree of the underlying undirected graph.} $\Delta$ we have that $\vv{r}(D) \leq c_\Delta |V(D)|$ for some constant $c_\Delta$ depending only on $\Delta$.
Quite surprisingly, Fox, He, and Wigderson \cite{fox2021ramsey} recently answered this question in the negative by showing that for any $\Delta$ and large enough $n$ there exists an $n$-vertex digraph $D$ with maximum degree $\Delta$ and $\vv{r}(D) \geq n^{\Omega(\Delta^{2/3} / \log^{5/3} \Delta)}$. 
In the other direction, they proved an upper bound of $\vv{r}(D) \leq n^{C_\Delta \log n}$ for any acyclic digraph $D$ with maximum degree $\Delta$. However, their results leave open the question of whether the worst case behavior for fixed $\Delta$ is always polynomial, or whether it can indeed be super-polynomial.

Why are the Ramsey numbers of acyclic digraphs so much larger than analogous Ramsey numbers in the undirected setting? There are a few reasons, but it seems that an important one has to do with chromatic numbers. In the undirected setting, a graph of bounded maximum degree also has bounded chromatic number, a fact which is often used in upper bounds on $r(G)$; this is also (roughly) why it is possible to prove stronger bounds if one assumes that $G$ is bipartite. However, in the directed setting, the correct analogue of the chromatic number seems to be the \emph{height} of $D$, which is defined as the number of vertices in the longest oriented path in $D$, or equivalently as the minimum number of parts in a partition $V(D)=V_1 \cup \dotsb \cup V_h$ with the property that all edges are directed from $V_i$ to $V_j$ for some $i<j$. As shown by simple examples such as oriented paths, acyclic digraphs can have bounded maximum degree but unbounded height, and this manifests itself in the fact that acyclic digraphs can have much larger Ramsey numbers than their undirected counterparts. In fact, Fox, He, and Wigderson \cite{fox2021ramsey} proved that if $D$ has maximum degree $\Delta$ and height $h$, then $\vv r(D) \leq C_{\Delta,h}n$ for some constant $C_{\Delta,h}$ depending only on $\Delta$ and $h$, showing that an analogue of the Burr--Erd\H os conjecture is true for digraphs of bounded height.
However, it is not the case that $D$ having large height necessarily implies that $\vv r(D)$ is large. Instead, it appears that what really matters is the \emph{multiscale complexity} of $D$---roughly speaking, how many edge length scales $D$ has (we refer to \cite[Section 1]{fox2021ramsey} for more details).
In particular, Fox, He and Wigderson \cite{fox2021ramsey} showed that even if $D$ has unbounded height, its oriented Ramsey number is polynomial in $\ab{V(D)}$ if it satisfies certain other restrictions on the distribution of edges. An important class of digraphs for which they proved such a result is the class of \emph{graded} digraphs, which we now define.
\begin{definition}
    We say that a digraph $D$ is \emph{graded} with a \emph{graded partition} $V(D) = V_1 \cup \dots \cup V_h$ if every edge of $D$ is directed from $V_i$ to $V_{i+1}$ for some $i \in [h-1]$. 
\end{definition}
We remark that graded digraphs are necessarily acyclic, and that the graded partition is unique assuming that the underlying graph of $D$ is connected. In particular, the number $h$ of parts in the graded partition is equal to the height of $D$. 

There are many natural examples of graded digraphs. For example, the $d$-dimensional \emph{grid digraph} $\vv{\Gamma_{d,k}}$ whose vertex set is $[k]^d$ and whose edges are all ordered pairs of the form $$((x_1,\dots,x_i,\dots,x_d), (x_1,\dots,x_i+1,\dots,x_d)) \qquad \text{ for some }i \in [d]$$ is graded; one obtains the graded partition by setting $V_i \coloneqq \{(x_1,\dots,x_d) \in [k]^d : x_1+\dotsb+x_d=i\}$. An important special case of this construction is the \emph{oriented hypercube} $\vv{Q_d}$, which is obtained from the unoriented hypercube graph by directing all edges away towards the positive orthant. More generally, the Hasse diagram of any graded poset is a graded digraph.

Fox, He, and Wigderson \cite{fox2021ramsey} 
proved the following upper bound on oriented Ramsey numbers of graded digraphs.
\begin{theorem}[{\cite[Theorem 1.5]{fox2021ramsey}}]
    If $D$ is a graded digraph with $n$ vertices, height $h$, and maximum degree $\Delta$, then
    \[
    \vv r(D) \leq h^{10\Delta \log_2 \Delta}n.
    \]
    In particular, as $h \leq n$, we have $\vv r(D) \leq n^{11\Delta \log_2 \Delta}$.
\end{theorem}

\subsection{Main results}
In this paper, we show a much stronger bound on the oriented Ramsey number of graded digraphs. 
In fact, we show that for any fixed $\Delta$, any graded digraph with maximum degree $\Delta$ has an oriented Ramsey number that is linear in its order.
\begin{theorem}\label{theorem:easy_upper_bound}
    If $D$ is a graded digraph on $n$ vertices with maximum degree $\Delta$, then
    \[
        \vv r(D) \leq 10^9 \Delta^3 2^{4\Delta}n.
    \]
    More precisely, if $D$ has maximum in-degree $\Delta^-$ and maximum out-degree $\Delta^+$, then
    $$\vv{r}(D) \leq 10^9 \Delta^+ (\Delta^-)^2 2^{4\Delta^-} n.$$
\end{theorem}
We stress that the height of $D$ does not affect the bound in Theorem \ref{theorem:easy_upper_bound} at all. This is somewhat surprising, given the intuition above that the height of a digraph should play an analogous role to the chromatic number of a graph, and should in turn affect the oriented Ramsey number. Nonetheless, an understanding that arises from our techniques is that graded digraphs ``behave like'' bipartite graphs, regardless of their height. 

As an immediate corollary of Theorem \ref{theorem:easy_upper_bound}, we obtain a linear upper bound on the oriented Ramsey numbers of grid digraphs in any fixed dimension, since $\vv{\Gamma_{d,k}}$ has maximum in- and out-degree equal to $d$.
\begin{corollary}
    For any $d \geq 1$, there exists a constant $C_d = 10^9 d^3 2^{4d}$ such that the $d$-dimensional grid digraph $\vv{\Gamma_{d,k}}$  satisfies $\vv r(\vv{\Gamma_{d,k}}) \leq C_d \ab{V(\vv{\Gamma_{d,k}})}$.
\end{corollary}
We remark that there has recently been a great deal of interest in Ramsey- and Tur\'an-type questions involving grid graphs, see e.g.\  \cite{conlon2023size, clemens2021size, gao2023extremal,bradac2023turan,gishboliner2022constructing,furedi2013uniform,kim2016two,mota2015ramsey}. In particular, it is proved in \cite[Corollary 1.4]{mota2015ramsey} that the \emph{undirected} Ramsey number of the two-dimensional $k\times k$ grid graph is $(\frac 32+o(1))k^2$.

At the other extreme, where $k$ is fixed and $d$ tends to infinity, we obtain a polynomial bound. For example, for the oriented hypercube $\vv{Q_d}$, Theorem \ref{theorem:easy_upper_bound} implies that $\vv r(\vv{Q_d})\leq 2^{5d+o(d)} = \ab{V(\vv{Q_d})}^{5+o(1)}$. By optimizing our techniques, we are able to improve the exponent from $5$ to $\log_2(17) \approx 4.09$.
\begin{theorem}\label{thm:hypercube}
    There exists an absolute constant $C>0$ such that $\vv{r}(\vv{Q_d}) \leq C d^3 17^{d}$.
\end{theorem}
In fact, motivated by the example of the hypercube, in  \cref{section:upper_bound} we prove a strengthening of \cref{theorem:easy_upper_bound}, which gives a better bound for graded digraphs where the large in-degrees are only in the parts of the graded partition where the number of vertices is small. Such a result is useful for $\vv{Q_d}$, since in the hypercube, almost all vertices (and in particular those vertices lying in the very large parts of the graded partition) have in- and out-degree close to $d/2$.

In the undirected setting, it is a major open problem to determine $r(Q_d)$.
A famous conjecture of Burr and Erd\H{o}s \cite{burr_magnitude_1975} from 1975 is that the Ramsey number of the hypercube is linear in its order, i.e.\ that $r(Q_d) = O(2^d)$. This question has been intensively studied (see e.g.\ \cite{beck1983upper, graham2001bipartite, shi2001cube, shi2007tail, fox2009density, conlon2016short}); the current best known bound is due to Tikhomirov \cite{tikhomirov2024remark}, who proved that $r(Q_d) \leq 2^{(2 - \varepsilon)d}$, where $\varepsilon > 0$ is some small absolute constant.
However, the results in the undirected setting cannot be used directly to obtain upper bounds on $\vv r(\vv{Q_d})$, and to the best of our knowledge \cref{thm:hypercube} is the first known polynomial bound on $\vv r(\vv{Q_d})$. Instead, the techniques from the unidrected setting naturally yield polynomial upper bounds on the oriented Ramsey number of the \emph{bipartite orientation} of the hypercube, where all edges are directed from vertices of even to odd Hamming weight, rather than the more natural orientation $\vv{Q_d}$.

Our second main result shows that \cref{theorem:easy_upper_bound} is close to best possible, in the sense that there exist graded digraphs with maximum degree $\Delta$ and oriented Ramsey number of at least $c^\Delta |V(D)|$ for some absolute constant $c>1$. A very similar result was proved in the undirected setting by Graham, R\"odl and Ruci\'nski \cite{graham2001bipartite}; in fact, it is not hard to adapt their construction to show the existence of such a $D$ which is \emph{bipartite}, that is, of height $2$. However, by modifying their construction appropriately, we are able to prove such a result for $D$ of arbitrary height, where the lower bound again does not depend on the height.

\begin{theorem}\label{theorem:lower_bound}
    There exist constants $c > 1$ and $\Delta_0$ such that for all $\Delta \geq \Delta_0$, $n \geq \Delta$, and $h \geq 2$ there exists a graded digraph $D$ with maximum degree $\Delta$ and with a graded partition $V(D) = V_1 \cup \dots \cup V_h$ such that $|V_i| \leq n$ for all $i \in [h]$, such that $\vv{r}(D) > c^\Delta h n \geq c^\Delta \ab{V(D)}$.
\end{theorem}

While our proof of \cref{theorem:lower_bound} is fairly closely inspired by the techniques of Graham--R\"odl--Ruci\'nski \cite{graham2001bipartite}, there are a number of techniques that go into the proof of \cref{theorem:easy_upper_bound}. The three main ingredients are the dependent random choice technique, the median order of a tournament, and an embedding argument based on the Lov\'asz local lemma. Dependent random choice is a very powerful technique which has been used in many of the recent breakthroughs in the Ramsey theory of sparse graphs; we refer to the survey \cite{fox2011dependent} for more information. The median order is an elementary, yet surprisingly powerful, tool in the study of tournaments, and has been instrumental in many of the recent advances on Sumner's conjecture and related questions; see e.g.\ \cite{havet2000median} for more details. The idea of combining these two techniques\footnote{Strictly speaking, \cite{draganic2021powers} does not use dependent random choice, but rather a simpler K\H ov\'ari--S\'os--Tur\'an-type argument. Nonetheless, the high-level approach is closely related.} was already used in \cite{draganic2021powers}, where an upper bound on the oriented Ramsey number of the power of an oriented path is proved; at a high level, this digraph is embedded in an arbitrary tournament by repeatedly embedding small pieces, and each small piece is embedded by combining dependent random choice with properties of the median order. We apply a similar technique, but as we are embedding an arbitrary graded digraph, more care is needed in the embedding step. To achieve this, we embed each part of $D$ randomly, and ensure that the embedding is successful thanks to the Lov\'asz local lemma; this is inspired by a similar argument in \cite{conlon2016short}, upper-bounding the Ramsey numbers of bounded-degree bipartite graphs. However, as our digraphs may have arbitrary height, we need to ensure at every embedding step both that we are consistent with the past, and that we are well-equipped to apply the local lemma at the next step, which adds some further complications. We defer a more detailed proof sketch to \cref{sec:UB outline}.

The remainder of the paper is organized as follows.
We prove (a strengthening of) \cref{theorem:easy_upper_bound} in \cref{section:upper_bound}, which begins with a detailed proof outline. 
Similarly, in  \cref{section:lower_bound}, we first give a sketch of the proof of \cref{theorem:lower_bound}, then prove it for height-$2$ digraphs, and finally prove it in full generality. We end in \cref{sec:conclusion} with some concluding remarks and open problems.

\paragraph{Notation:} For a directed graph $D$, we use $V(D)$ to denote its vertex set and $E(D)$ to denote its edge set, which is a collection of ordered pairs of elements of $V(D)$.
For a vertex $v \in V(D)$ we write $N_D^+(v)$ and $N_D^-(v)$ for its out- and in-neighborhood and $d_D^+(v)$ and $d_D^-(v)$ for its out- and in-degree, respectively.
For a subset $U \subseteq V(D)$ we let $N_D^+(U)$ denote the common out-neighborhood of $U$ and similarly $N_D^-(U)$ the common in-neighborhood of $U$.
Additionally, for a vertex $v \in V(D)$ we let $d_D^+(v, U) = |N_D^+(v) \cap U|$ and $d_D^-(v, U) = |N_D^-(v) \cap U|$ be the out- and in-degree of $v$ into $U$.
Whenever the digraph $D$ is clear from context, we omit the subscript and write $d^+(v)$ for $d_D^+(v)$, etc.
Throughout the paper, we omit ceilings and floors whenever they are not crucial.

\section{Proof of Theorem \ref{theorem:easy_upper_bound}}\label{section:upper_bound}

As stated above, instead of proving Theorem \ref{theorem:easy_upper_bound}, we will prove a stronger theorem that leverages the fact that the graded digraph could locally have different maximum in-degrees in different parts.
This will allow us to use less overhead to embed the parts where the in-degree is small, and is particularly useful for graded digraphs in which the large in-degree only appears in parts of the graded partition whose sizes are small, as is the case for the oriented hypercube $\vv{Q}_d$.
More formally, we prove the following theorem, which immediately implies \cref{theorem:easy_upper_bound} by upper-bounding each $\Delta_i^-$ by $\Delta^-$.
\begin{theorem}\label{theorem:upper_bound}
    Let $D$ be a graded digraph on $n$ vertices with a graded partition $V(D) = V_1 \cup \dots \cup V_h$ for some $h \in \mathbb{N}$ and maximum in- and out-degree $\Delta^-$ and $\Delta^+$ respectively. 
    Moreover, for each $i\in [h-1]$ let $\Delta^-_i$ be the maximum in-degree in the induced subgraph $D[V_i \cup V_{i+1}]$ and set $\Delta^-_0 = \Delta^-_{h} = 0$.
    Then $$\vv{r}(D) \leq 10^9 (\Delta^-)^2 \Delta^+ \sum_{i=1}^h 2^{2(\Delta_{i-1}^- + \Delta_{i}^-) } |V_i|.$$
\end{theorem}
Before proceeding with the proof of \cref{theorem:upper_bound}, we note that it immediately yields \cref{thm:hypercube}, the upper bound on the oriented Ramsey number of the hypercube.
\begin{proof}[Proof of  \cref{thm:hypercube}]
    Let $d \in \mathbb{N}$ and $\vv{Q_d}$ be the $d$-dimensional hypercube on the vertex set $V = \{ 0,1 \}^d$.
    Moreover, for each $0\leq i \leq d$, let $V_i := \{ v\in V : \sum_j v_j = i \}$.
    Then $V = V_0 \cup \dots \cup V_d$ is a graded partition of $\vv{Q_d}$ and for each $i = 0, \dots, d-1$ the maximum in-degree in the induced subgraph $\vv{Q_d}[V_i \cup V_{i+1}]$ is $i+1$.
    Therefore, by Theorem \ref{theorem:upper_bound},
    \[
        \vv{r}(\vv{Q_d}) \leq C' d^3 \sum_{i=0}^d \binom{d}{i} 4 \cdot 2^{4i} = C d^3 (16 + 1)^d = C d^3 17^d,
    \]
    for some absolute constants $C'$ and $C$.
\end{proof}

\subsection{Proof outline}\label{sec:UB outline}
We now sketch our proof of \cref{theorem:upper_bound}, before proceeding with the technical details.
As our goal is upper-bounding $\vv r(D)$, we wish to prove that any sufficiently large tournament contains a copy of $D$, so we fix such a tournament $T$.
We will first find disjoint subsets $A_1, \dots, A_h$ such that almost all $(\Delta_{i-1}^-)$-subsets of $A_{i-1}$ have many out-neighbors in $A_{i}$.
To achieve this, we will use the notion of a median order and then apply the dependent random choice technique.
A \emph{median order} is an ordering $v_1, \dots, v_N$ of vertices of $T$ that maximizes the number of forward edges, that is, edges $v_iv_j$ with $i < j$.
Given a median order $v_1, \dots, v_N$ we write $[i, j)$ for $\{v_i , \dots, v_{j-1} \}$ and similarly $[i, j]$ for $\{v_i, \dots, v_j\}$.
We note that in a median order for any $v_i$ and $j < i$ at least half of the vertices $[j, i)$ are in-neighbors of $v_i$, since otherwise we could move $v_i$ to the front of the interval to obtain an ordering with more forward edges.

To find our sets $A_1, \dots, A_h$, we will place $A_h$ at the end of a median order of $T$ and then inductively find $A_{h-1}, A_{h-2}$, and so on, all the way down to $A_1$.
At each step, we will make sure that the last $A_i$ found is contained within some small interval $I_i$ in the median order.
To find $A_{i-1}$ we will first find an interval $I_{i-1}$ so that almost half of the edges go from $I_{i-1}$ to $A_i$.
Using properties of the median order, we can always find $I_{i-1}$ within a bigger interval $I'_{i-1}$, which immediately precedes $I_i$ and is only larger than $I_i$ by a fixed multiplicative factor. However, we will also ensure that $\ab{I_{i-1}}=\ab{I_i}$, so that
this multiplicative factor does not stack up as we proceed with the next layers. Given such $I_{i-1}$ and $A_i$, a standard application of the dependent random choice technique allows us to construct $A_{i-1}\subseteq I_i$ with the property that most $(\Delta_{i-1}^-)$-subsets of $A_{i-1}$ have many common out-neighbors in $A_{i}$.

Having found all the sets $A_1, \dots, A_h$, we now want to embed each $V_i$ into the corresponding $A_i$.
We will start this time with $V_1$ and after having embedded the first $i-1$ layers we will use the Lov\'asz Local Lemma to embed the $i$th layer.
The added difficulty in this step is that we wish to maintain the ability to keep this process going;
therefore, when embedding the $i$th layer, we will also make sure that each vertex of the $(i+1)$st layer we will still have many possible embeddings in $A_{i+1}$.
This allows us to then reuse the local lemma argument in the next layer as well.

We now proceed with the details of the proof, which is split across the next three subsections.
The basic dependent random choice lemma, which we will iterate to construct the sets $A_i$, is stated and proved in \cref{section:dependent_random_choice}.
The basic embedding lemma, which allows us to embed $V_i$ into $A_i$ while being in the position to continue the process in step $i+1$, is given in \cref{section:finding_embedding}. 
We combine these two tools in \cref{section:proof_upper_bound}, where we prove \cref{theorem:upper_bound} via two inductive arguments: first to find the sets $A_1, \dots, A_h$, and then the second to embed our digraph $D$ into them.
While this is essentially just a combination of the two main lemmas, some care is needed to ensure that
the dependencies between the various parameters work out properly.

\subsection{Dependent random choice}\label{section:dependent_random_choice}
In this section, we prove a lemma that will allow us to find the subsets $A_1, \dots, A_h$ as described above. While the statement may seem daunting because of all of the parameters, the basic idea is the simple one described above: given a set $B$ contained in an interval of the median ordering, we are able to find a set $A$, contained in a nearby interval of the same size, such that $A$ is large and such that most $(\Delta^-)$-subsets of $A$ have many common out-neighbors in $B$.
\begin{lemma}\label{lemma:dependent_random_choice}
    Let $a, a', b, \ell, s, N, k, \Delta^- \in \mathbb{N}$ be integers with $2a' \geq a$, and let $T$ be a tournament on $N$ vertices with a median order $v_1, \dots, v_N$.
    Moreover, let $2ka' < j \leq N - a + 1$ and let $B \subseteq [j, j+a)$ be an arbitrary subset of size at least $b$.
    Then there exist $j - 2ka' \leq j' \leq j-a'$ and $A \subseteq [j', j' + a')$ such that
    \begin{itemize}
        \item $|A| \geq \frac{a'}{2} (\frac{k-1}{2k})^\ell,$ {and}
        \item for all but at most $4(\frac{2k}{k-1})^\ell\binom{a'}{\Delta^-}(\frac{s}{b})^\ell$ subsets $S \subseteq A$ with $|S| = \Delta^-$, we have $|N^+(S) \cap B| \geq s+1.$
    \end{itemize}
\end{lemma}
\begin{proof}
    Note that we may assume $\ab B = b$, because the conclusion for any larger value of $\ab B$ is a strictly weaker statement. 
    Let $J = [j-2ka', j)$ and notice that since $v_1, \dots, v_N$ is a median order of the vertices of $T$ at least half of the vertices from $[j-2ka', i)$ are in-neighbors of $v_i$ for each $v_i$. Therefore, for every $v_i \in [j,j+a)$, we have
    \[
        d^-(v_i,J) = d^-(v_i, [j-2ka',i)) - d^-(v_i, [j,i)) \geq \frac{2ka'+i-j}{2} - (i-j)=\frac{2ka'-(i-j)}{2} \geq (k-1)a',
    \]
    where the final step uses that $i-j \leq a \leq 2a'$.
    Therefore, since $B \subseteq [j, j+a)$, we have
    \[
        \sum_{u \in J} d^+(u, B) = \sum_{v\in B} d^-(v, J) \geq b(k-1)a'.
    \]
     Now, for $i \in [2k]$ let $I_i = [j+ (i - 2k - 1)a', j+(i-2k)a')$ and notice that $J = \bigcup_{i \in [2k]} I_i$.
     Therefore, by the pigeonhole principle there exists an $i$ such that
     \[
        \sum_{u \in I_i} d^+(u, B) \geq b\left(\frac{k-1}{2k}\right)a'.
     \]
    Fix such an $i$, and let $I = I_i$. Let $j'=j+(i-2k-1)a'$ be the left endpoint of $I$.
    
    Pick now a set $L$ of $\ell$ vertices from $B$ uniformly at random with repetitions, let $M = \{ u \in I : L \subseteq N^+(u) \cap B\}$, and write $X = |M|$.
    Jensen's inequality then implies that
    \[
        \mathbb{E}[X] = \sum_{u \in I} \left(\frac{d^+(u,B)}{b}\right)^\ell \geq 
        a' \cdot \left(\frac{k-1}{2k}\right)^\ell.
    \]
    Let $Y$ be the random variable counting the number of subsets of $M$ of size $\Delta^-$ with at most $s$ common out-neighbors in $B$.
    For a given such $S \subseteq I$, the probability that it is a subset of $M$ is at most $(\frac{\ab{N^+(S)\cap B}}{b})^\ell \leq (\frac{s}{b})^\ell$ and therefore
    \[
        \mathbb{E}[Y] \leq \binom{a'}{\Delta^-}\left(\frac{s}{b}\right)^\ell.
    \]
    Let us suppose for the moment that $\mathbb{E}[Y] > 0$. By linearity of expectation we know that
    \[
        \mathbb{E}\left[X - \frac{\mathbb{E}[X]}{2\mathbb{E}[Y]}Y - \frac{\mathbb{E}[X]}{2}\right] = 0.
    \]
    Therefore, there exists a choice of $L$ and the corresponding $M$ for which this expression is non-negative. Fix such a choice.
    Then
    \[
        \ab M=X \geq \frac{\mathbb{E}[X]}{2} \geq \frac{a'}{2} \cdot \left(\frac{k-1}{2k}\right)^\ell,
    \]
    and, since $X \leq |I| = a'$
    \[
        Y \leq \frac{2X}{\mathbb{E}[X]}\mathbb{E}[Y] \leq 4 \left(\frac{2k}{k-1}\right)^\ell \binom{a'}{\Delta^-} \left(\frac{s}{b}\right)^\ell.
    \]
    In particular, we may set $A=M$ and conclude the proof.

    It remains to consider the case that $\mathbb{E}[Y] = 0$. In this case, as $Y$ is a non-negative random variable, it must take on the value $0$ with probability $1$. Thus, if we select any $L$ for which the corresponding $X$ is at least $\mathbb E[X]$, we can again set $A=M$ and conclude the proof.
\end{proof}

\subsection{Finding a good embedding}\label{section:finding_embedding}
In this section we prove a lemma that will allow us to embed each part $V_i$ of the graded partition into the respective $A_i$.
For this, we will use the Lov\'{a}sz local lemma \cite{lovasz_lemma}, whose statement we now recall.
\begin{lemma}\label{lemma:lovasz_local_lemma}
    Let $A_1, \dots, A_n$ be events in an arbitrary probability space and let $H = ([n], E)$ be a graph such that for each $i\in[n]$ the event $A_i$ is mutually independent of the events $\{ A_j : (i, j) \notin E\}$.
    Suppose moreover that $0 \leq x_1, \dots, x_n < 1$ are real numbers such that for all $i \in [n]$ we have $\Pr[A_i] \leq x_i \prod_{(i, j) \in E} (1 - x_j)$.
    Then
    \[
        \Pr\left[\bigwedge_{i=1}^{n} \overline{A_i}\right] \geq \prod_{i=1}^n (1 - x_i) > 0.
    \]
\end{lemma}

We are now ready to state and prove our embedding lemma. At a high level, it says the following. We wish to embed a bipartite digraph $D=(V_1 \cup V_2,E)$ into a tournament $T$. Eventually, the bipartite digraph will be the induced subdigraph on two consecutive layers of the graded partition. As such, we have target sets $A,B \subseteq V(T)$, and would like to embed the $V_1$ into $A$ and $V_2$ into $B$. Eventually, these sets will come from our list $A_1,\dots,A_h$, so we may assume that most $(\Delta^-)$-subsets of $A$ have many common neighbors in $B$, which is of course convenient for successfully embedding $D$. However, because we wish to be consistent with the previous steps of the embedding, we cannot embed $V_1$ arbitrarily into $A$; rather, we are given a list of options $f(v) \subseteq A$ for each vertex $v \in V_1$, and must embed $v$ into one of the options from $f(v)$. Moreover, we want to ensure that, having embedded $V_1$, there are actually many ways to embed each vertex of $V_2$ into $B$; this is necessary to ensure that at the next stage of the embedding, the lists $f(v)$ remain sufficiently large to apply the same lemma again. The next lemma simply states that, as long as the relevant parameters are related in the appropriate ways, such an embedding is possible. It is proved by embedding each $v$ into a uniformly random choice in $f(v)$, and using the local lemma to show that we succeed with positive probability.

\begin{lemma}\label{lemma:embedding}
    Let $D = (V_1 \cup V_2, E)$ be a bipartite digraph with maximum in- and out-degree $\Delta^-\geq 1$ and $\Delta^+$, respectively, with all edges oriented from $V_1$ to $V_2$.
    Let $a, b, c \in \mathbb{N}$ be integers and $\delta>0$ be a parameter such that $a \geq b \geq 32\ab{V_1}$ and $\delta \leq (2^{-1/2}\frac{b}{a})^{\Delta^-}/({4\Delta^+ \Delta^-})$.
    Let $T$ be a tournament and $A, B \subseteq V(T)$ be subsets of its vertices such that $|A| = a$ and for all but at most a $\delta$-fraction of subsets $S \subseteq A$ of size $\Delta^-$ we have $|N_T^+(S) \cap B| \geq c$.

    Let $f: V_1 \to 2^{A}$ be a function such that for each $v \in V_1$ we have $|f(v)| \geq b$. Then there exists an injective function $\phi: V_1 \to A$ such that for each $v \in V_1$ we have $\phi(v) \in f(v)$ and for each $u \in V_2$ we have $|N_T^+(\phi(N_D^-(u)))\cap B| \geq c$.
\end{lemma}
\begin{proof}
    Let $s = |V_1|$ and let $\phi: V_1 \to A$ be a random mapping such that for each $v \in V_1$ the value of $\phi(v)$ is picked uniformly at random from the set $f(v)$, with all these choices made independently.
    For all distinct $v, w  \in V_1$, let $A_{vw}$ be the event that $\phi(v) = \phi(w)$.
    Moreover, for each $u \in V_2$ let us write $N_u = N_D^-(u)$ and let $B_u$ be the event that $|\phi(N_u)| = N_u$ but $|N^+_T(\phi(N_u)) \cap B| < c$.
    Clearly, if none of the bad events $A_{vw}$ and $B_u$ hold, then $\phi$ satisfies the requirements of the lemma.

    Let us now bound the probabilities for each of these events to hold.
    For $v, w \in V_1$ we have $\Pr[A_{vw}] \leq \frac{1}{b}$.
    To bound $\Pr[B_u]$, let us suppose that $|\phi(N_u)| = |N_u|$ but $|N^+_T(\phi(N_u)) \cap B| < c$. If we denote $\ell=\ab{N_u}$, this implies that
    $|N_T^+(S) \cap B| < c$ for all $\binom{a-\ell}{\Delta^- - \ell}$ $(\Delta^-)$-subsets $S$ of $A$ containing $\phi(N_u)$.
    Moreover, the number of pairs consisting of a bad $(\Delta^-)$-subset $S \subseteq A$ and an $\ell$-subset of $S$ is at most $\delta \binom{a}{\Delta^-}\binom{\Delta^-}{\ell}$.
    Therefore, the number of tuples $(\phi(v))_{v \in N_u} \in A^\ell$ for which the event $B_u$ holds is at most
    \[
        \frac{\delta \binom {a}{\Delta^-}\binom{\Delta^-}{\ell}\ell!}{\binom{a-\ell}{\Delta^- - \ell}}  = \delta \binom{a}{\ell} \ell! \leq \delta a^\ell.
    \]
    Since $(\phi(v))_{v \in N_u}$ is chosen uniformly at random from a subset of $A^\ell$ of size at least $b^\ell$, this implies that
    $\Pr[B_u] \leq \frac{\delta a^\ell}{b^\ell} \leq \delta (\frac{a}{b})^{\Delta^-}$.

    Let us now consider the dependencies between the bad events.
    Note that since the random variables $\{\phi(v)\}_{v \in V_1}$ are mutually independent, the event $A_{vw}$ is mutually independent from all $A_{v'w'}$ and $B_u$ such that $\{v, w\} \cap \{v', w'\} = \varnothing$ and $\{v, w \} \cap N_u = \varnothing$, respectively.
    Thus, $A_{vw}$ is dependent on at most $2(s-2)<2s$ events $A_{v'w'}$ and at most $2\Delta^+$ events $B_u$.
    Similarly, the event $B_u$ is mutually independent from all $A_{v,w}$ and $B_{u'}$ such that $N_u \cap \{v, w\} = \varnothing$ and $N_{u} \cap N_{u'} = \varnothing$, respectively, and thus it is dependent on at most $\binom{s}{2} - \binom{s - |N_u|}{2} \leq \frac{2s\Delta^- - (\Delta^-)^2 - \Delta^-}{2} < s\Delta^-$ events $A_{vw}$ and at most $\Delta^+\Delta^-$ events $B_{u'}$.

    We now want to apply the local lemma.
    For each $A_{vw}$, we let the corresponding $x_i$ be $x = \frac{8}{b}$ and for each $B_u$ we let the corresponding $x_i$ be $y = \frac{1}{2\Delta^+\Delta^-}$. Since $\Delta^- \geq 1$ and $b \geq 32s$ we then get
    \[
        x(1-x)^{2s}(1-y)^{2\Delta^+} \geq \frac{8}{b} 4^{-16s/b} 4^{-1/\Delta^-} \geq \frac{1}{b} \geq \Pr[A_{vw}],
    \]
    where we use the inequality $1-z \geq 4^{-z}$, valid for all $0 \leq z \leq \frac 12$. Similarly, we have
    \[
        y(1-y)^{\Delta^+ \Delta^-}(1 - x)^{s\Delta^- } \geq \frac{1}{2\Delta^+\Delta^-} 4^{-1/2}4^{-8s\Delta^-/b} \geq  \frac{1}{4\Delta^+\Delta^-} 4^{-\Delta^-/4}
        = \left(\frac{a}{b}\right)^{\Delta^-} \delta \geq \Pr[B_u],
    \]
    where the equality is by our choice of $\delta$.
    Therefore, by \cref{lemma:lovasz_local_lemma}, the probability that none of the bad events $A_{vw}$ and $B_u$ hold is positive and thus, there exists a choice of $\phi$ satisfying the desired properties.
\end{proof}

\subsection{Proof of Theorem \ref{theorem:upper_bound}}\label{section:proof_upper_bound}
Given \cref{lemma:dependent_random_choice,lemma:embedding}, we are now ready to prove \cref{theorem:upper_bound}. We recall that the high-level idea of the proof is to inductively construct sets $A_h,\dots,A_1$ using \cref{lemma:dependent_random_choice}, and then to embed $D$ into these sets inductively using \cref{lemma:embedding}. Essentially all that remains is to define a large number of parameters, and to check that they satisfy certain inequalities so that \cref{lemma:dependent_random_choice,lemma:embedding} can indeed be applied.
\begin{proof}[Proof of Theorem \ref{theorem:upper_bound}]
    Let $D$ be a graded graph on $n$ vertices with a graded partition $V(D) = V_1 \cup \dots \cup V_h$ for some $h \in \mathbb{N}$ and let $\Delta^+$ and $\Delta^-$ be its maximum out- and in-degree, respectively.
    Moreover, for each $i \in [h-1]$ let $\Delta_i^-$ be the maximal in-degree in the induced subgraph $D[V_i \cup V_{i+1}]$ and set $\Delta^-_0 = \Delta^-_h = 0$. Note that we may assume that $\Delta_i^-\geq 1$ for all $i \in [h-1]$, for otherwise the underlying graph of $D$ is disconnected, and we obtain the desired bound on $\vv r(D)$ by summing up $\vv r(D')$ for every connected component $D'$ of $D$.

    We define $\varepsilon = 2/\Delta^-$ and $k = 4\Delta^- + 4$, and note that $\frac{2k}{k-1} \leq 2 + \varepsilon$.
    To find the subsets $A_h, \dots, A_1$ as described above we first need to fix the relevant parameters.
    We define integers $c_s$, $c_b$, and $c_a$ by
    \[
        c_s = 32, \qquad c_b = 2000 \Delta^+ \Delta^- c_s, \qquad\text{and} \qquad c_a = 2c_b.
    \]
    We further define
    \[
        n_i = \sum_{j=i}^h \frac{(2 + \varepsilon)^{2 (\Delta_{j-1}^- + \Delta_j^-)} |V_j|}{2^{j-i}},
    \]
    and let
    \[
        a_i = c_a n_i \qquad \text{ and } \qquad b_i = c_b (2 + \varepsilon)^{-2\Delta_i^-}n_i.
    \]
    In building the sets $A_h,\dots,A_1$, we will ensure that $\ab{A_i}\geq b_i$ and that $A_i$ lies in an interval of length $a_i$. We now let
    \[
        s_i = c_s (2 + \varepsilon)^{-2(\Delta_i^- + \Delta_{i-1}^-)} n_i \qquad \text{ and } \qquad \delta_i =\frac{1}{4\Delta^- \Delta^+} \left(2^{-1/2}\frac{s_{i+1}}{b_{i+1}}\right)^{\Delta_i^-} .
    \]
    We will further guarantee that all but a $\delta_i$-fraction of the $(\Delta_i^-)$-subsets of $A_i$ have at least $s_{i+1}$ common out-neighbors in $A_{i+1}$. Finally, we let
    \[
        o_i = 2k \sum_{j=i}^h a_i.
    \]
    Throughout the process, we will ensure that $A_i$ lies within the last $o_i$ vertices of the median order.
    We note for future reference that 
    \[
        n_i = (2+\varepsilon)^{2(\Delta_{i-1}^-+\Delta_i^-)}\ab{V_i} + \frac 12 n_{i+1}.
    \]
    In particular, this implies that $a_i \geq a_{i+1}/2$ and that $s_i \geq c_s \ab{V_i}= 32|V_i|$.
    
    Let $N = o_1$, let $T$ be a tournament on $N$ vertices, and fix a median order $v_1, \dots, v_N$ of $T$.
    We now claim that we can find integers $j_1,\dots,j_h$ and disjoint sets $A_1,\dots,A_h \subseteq V(T)$ satisfying the following properties.
    \begin{itemize}
        \item $j_i < j_{i+1}$,
        \item $j_i \geq N - o_i$,
        \item $A_i \subseteq [j_i, j_i + a_i)$,
        \item $|A_i| \geq b_i$, {and}
        \item for each $i \in [h-1]$, for all but at most a $\delta_i$-fraction of subsets $S \subseteq A_i$ of size $\Delta_i^-$, we have that $|N_T^+(S) \cap A_{i+1}| \geq s_{i+1}$.
    \end{itemize}
    We start by setting $j_h = a_h$ and $A_h = [N - a_h + 1, N]$, which clearly satisfy these properties.
    Suppose now that for some $i \in [h-1]$ we have defined $A_{i'}$ and $j_{i'}$ for all $i < i' \leq h$.
    By applying  \cref{lemma:dependent_random_choice} with $j = j_{i+1}$, $B = A_{i+1}$, $a = a_{i+1}, a' = a_i \geq a_{i+1}/2, b = b_{i+1}, \ell = 2\Delta^-_i, s = s_{i+1}, k=k$ and $\Delta^- = \Delta_i^-$ we can find $j_{i+1} > j_i \geq j_{i+1} - 2ka' \geq N - o_{i+1}$ and $A_i \subseteq [j_i, j_i + a_i)$ such that
    \begin{itemize}
        \item $|A_i| \geq \frac{a_i}{2}(2+ \varepsilon)^{-2\Delta^-_i} = b_i$, {and}
        \item at most $4 \cdot (2+\varepsilon)^{2\Delta_i^-}\binom{a_i}{\Delta_i^-}(\frac{s_{i+1}}{b_{i+1}})^{2\Delta_i^-}$ subsets $S \subseteq A_i$ of size $\Delta_i^-$ have fewer than $s_{i+1}$ common out-neighbors in $A_{i+1}$.
    \end{itemize}
    We now note that, since $b_i \geq 2\Delta_i^-$, we have that
    \[
        \frac{\binom{a_i}{\Delta_i^-}}{\binom{b_i}{\Delta_i^-}}\leq 2^{\Delta_i^-} \left( \frac{a_i}{b_i} \right)^{\Delta_i^-} =  \left(2\cdot \frac{c_a(2+\varepsilon)^{2\Delta_i^-}}{c_b} \right)^{\Delta_i^-} = \left( 4(2+\varepsilon)^{2\Delta_i^-} \right)^{\Delta_i^-},
    \]
    where we plug in our definitions of $a_i,b_i,c_a$, and $c_b$. 
    Additionally, we have that
    \[
        \left( \frac{s_{i+1}}{b_{i+1}} \right)^{\Delta_i^-}= \left( \frac{c_s(2+\varepsilon)^{2\Delta_{i+1}^-}}{c_b(2+\varepsilon)^{2(\Delta_{i+1}^-+\Delta_{i}^-)} } \right)^{\Delta_i^-} = \left( \frac{c_s}{c_b (2+\varepsilon)^{2\Delta_{i}^-}} \right)^{\Delta_i^-} = \left(\frac{1}{2000\Delta^+\Delta^-(2+\varepsilon)^{2\Delta_i^-}}\right)^{\Delta_i^-}
    \]
    by our choices of $s_{i+1},b_{i+1}, c_s$, and $c_b$.
    Putting this all together, we see that the fraction of the $(\Delta_i^-)$-subsets of $A_i$ for which the common out-neighorhood is too small is at most
    \begin{align*}
        \frac{4 \cdot (2+\varepsilon)^{2\Delta_i^-}\binom{a_i}{\Delta_i^-}(\frac{s_{i+1}}{b_{i+1}})^{2\Delta_i^-}}{\binom{b_i}{\Delta_i^-}} &= \left[4(2+\varepsilon)^{2\Delta_i^-} \left( \frac{s_{i+1}}{b_{i+1}} \right)^{\Delta_i^-}\right]\cdot \left[ \frac{\binom{a_i}{\Delta_i^-}}{\binom{b_i}{\Delta_i^-}}\left( \frac{s_{i+1}}{b_{i+1}} \right)^{\Delta_i^-} \right]\\
        &\leq \left[4(2+\varepsilon)^{2\Delta_i^-} \left( \frac{s_{i+1}}{b_{i+1}} \right)^{\Delta_i^-}\right] \left( \frac{1}{500\Delta^+\Delta^-} \right)^{\Delta_i^-}\\
        &\leq \left( 2^{-1/2}\frac{s_{i+1}}{b_{i+1}} \right)^{\Delta_i^-} \left( \frac{4\cdot 2^{1/2}(2+\varepsilon)^2}{500\Delta^+\Delta^-} \right)^{\Delta_i^-}\\
        &\leq \delta_i,
    \end{align*}
    where the final step uses that $\varepsilon\leq 2$, that $\Delta_i^-\geq 1$, and our definition of $\delta_i$. This shows that the set $A_i$, as defined above, satisfies the desired properties. Continuing inductively in this way,
    we are able to find all the sets $A_1, \dots, A_h$.

    Having found these sets we now want to embed each $V_i$ into $A_i$, starting this time with $V_1$.
    For $i \in [h]$, let $S_i = \bigcup_{j =1}^i V_j$. For each $i \in [h]$ we find a function $\phi_i: S_i \to (A_1 \cup \dots \cup A_i)$ such that
    \begin{itemize}
        \item $\phi_i$ is an embedding of $D[S_i]$ into $T[A_1 \cup \dots \cup A_i]$,
        \item $\phi(V_j) \subseteq A_j$ for each $1 \leq j \leq i $, {and}
        \item for each $i \in [h-1]$ and $v \in V_{i+1}$, the vertices $\phi_i(N^-_D(v)) \subseteq A_i$ have at least $s_{i+1}$ common out-neighbors in $A_{i+1}$.
    \end{itemize}

    For $i=1$, we find $\phi_1$ by applying \cref{lemma:embedding} with $a = b = b_1$, $A = A_1$, $B = A_2$, $D = D[V_1 \cup V_2]$, $\delta = \delta_1$ and $f(v) = A_1$ for all $v \in V_1$.
    Suppose now that for some $i \in [h-2]$ we have found $\phi_i$ satisfying the conditions above.
    For $v \in V_{i+1}$, let $f(v) = N_T^+(\phi_i(N_D^-(v))) \cap A_{i+1}$ so that $|f(v)| \geq s_{i+1}$ for all $v \in V_{i+1}$.
    We can now again apply \cref{lemma:embedding} with $A = A_{i+1}$, $B = A_{i+2}$, $a = b_{i+1}$, $b = s_{i+1}$, $\delta = \delta_i$ and $D = D[V_{i+1} \cup V_{i+2}]$.
    We get a function $\phi$ such that
    \begin{itemize}
        \item for all $v \in V_{i+1}$ we have $\phi(v) \in N_T^+(\phi_i(N_D^-(v)))$, {and}
        \item for all $v \in V_{i+2}$ we have $|N_T^+(\phi(N_D^-(v)))| \geq s_{i+2}$.
    \end{itemize}
    Thus, the function
    \[
        \phi_{i+1}(v) \coloneqq \begin{cases}
            \phi(v), & v \in V_{i+1} \\
            \phi_i(v), & v \in S_i
        \end{cases}
    \]
    satisfies the conditions above.
    
    Proceeding in this way inductively, we can therefore find $\phi_{h-1}$ satisfying the same properties.
    Now, since for all $v \in V_h$ we have $|N_T^+(\phi_{h-1}(N_D^-(v)))| \geq s_h \geq |V_h|$ we can greedily extend $\phi_{h-1}$ into $\phi_h$ satisfying the above conditions. In particular, $\phi_h$ is an embedding of $D$ into $T$.
    
    To complete the proof, it remains to estimate $N$, the number of vertices in $T$.
    Plugging in our choice of $\varepsilon = 2/\Delta^-$, as well as the estimates $1+x\leq e^x$ and $\Delta_j^-\leq \Delta^-$ for all $j$, we find that
    \[
        (2+\varepsilon)^{2\Delta_{j-1}^- + 2\Delta_{j}^-} = (2(1+\varepsilon/2))^{{2\Delta_{j-1}^- + 2\Delta_{j}^-}} \leq  e^{4} 2^{2\Delta_{j-1}^- + 2\Delta_{j}^-}.
    \] 
    Additionally, we have that
    \begin{align*}
        \sum_{i=1}^h n_i &= \sum_{i=1}^h \sum_{j=i}^h \frac{(2 + \varepsilon)^{2 (\Delta_{j-1}^- + \Delta_j^-)} |V_j|}{2^{j-i}} \\
        &= 
        \sum_{j=1}^h (2 + \varepsilon)^{2 (\Delta_{j-1}^- + \Delta_j^-)} |V_j| \sum_{i=1}^j \frac 1{2^{j-i}}\\
        &\leq 2 \sum_{j=1}^h (2 + \varepsilon)^{2 (\Delta_{j-1}^- + \Delta_j^-)} |V_j|.
    \end{align*}
    Therefore,
    \[
        o_1 = 2k \sum_{i=1}^h a_i = 2kc_a \sum_{i=1}^h n_i \leq 4kc_a \sum_{j=1}^h (2+\varepsilon)^{2\Delta_{j-1}^- + 2\Delta_{j}^-} |V_j| \leq c (\Delta^-)^2 \Delta^+ \sum_{j=1}^h 2^{2(\Delta_{j-1}^- + \Delta_{j}^-) } |V_j|,
    \]
    for some absolute constant $c=4\cdot 5\cdot 32 \cdot 2000 \cdot 2\cdot e^4 \leq 10^9$.
    Since $T$ was an arbitrary tournament on $N=o_1$ vertices we have shown that 
    \[
        \vv{r}(D) \leq 10^9 (\Delta^-)^2 \Delta^+ \sum_{i=1}^h 2^{2(\Delta_{i-1}^- + \Delta_{i}^-) } |V_i|.\qedhere
    \]
 \end{proof}

\section{Proof of Theorem \ref{theorem:lower_bound}}\label{section:lower_bound}
\subsection{Proof outline}
In this section, we prove \cref{theorem:lower_bound}, which states that there exists a graded digraph $D$ with height $h$, maximum degree $\Delta$, and equal number of vertices in each part of the graded partition, such that $\vv r(D) \geq c^\Delta \ab{V(D)}$ for an absolute constant $c>1$.

The main ingredient of the proof is the same statement in the case $h=2$, from which the general statement will follow.
In other words, we first show that exists a bipartite digraph $D_0$ with vertex classes of size $n$ each and maximum degree at most $\Delta$ such that $\vv{r}(D_0) \geq c^\Delta \cdot(2n)$.

Having found such a bipartite digraph, we will generalize the construction for any height $h$, by taking a graded digraph $D$ such that the induced subgraph between two neighboring parts of the graded partition is a copy of $D_0$.
We will show that such a $D$ is not contained in a tournament $T$ obtained by replacing each vertex of a transitive tournament on $H = h/2 - 1$ vertices with a copy of $R$, a large tournament not containing $D_0$.

Indeed, if we let $D$ have the graded partition $V(D) = V_1 \cup \dots \cup V_h$ and let $V(T) = A_1 \cup \dots \cup A_{H}$ such that each $T[A_i]$ is a copy of $R$ and all the edges go from $A_i$ to $A_j$ for $i < j$, then we show the following.
If an embedding of $D$ into $T$ exists, then for any $i$, if we embedded most of $V_i$ into $A_j \cup \dots \cup A_H$ for some $j$, then most of $V_{i+2}$ must be embedded into $A_{j+1} \cup \dots \cup A_H$.
This will give us a contradiction since there are only $H < h/2$ levels in the tournament $T$.

It remains to construct $D_0$ and $R$ such that the argument above works. We will use a construction very similar to the one of Graham, Rödl and Ruciński \cite{graham2001bipartite}. Namely, we show that if we take $R$ to be a blow-up of a random tournament and $D_0$ a sparse random bipartite digraph, then with positive probability $R$ does not contain a copy of $D_0$.
To make the generalization for any height possible, we will in fact need to show a slightly stronger statement, namely that if we take any two large subsets $A'$ and $B'$ of the two vertex classes of $D_0$, then $R$ does not contain a copy of the induced subgraph $D_0[A' \cup B']$.

The remainder of this section contains the details of the argumetnt. We first construct a suitable bipartite digraph $D_0$ and a suitable tournament $R$ in \cref{section:guest_graph,section:host_graph} respectively.
Then in  \cref{section:bipartite_case} we show that $R$ indeed does not contain a copy of $D$.
Finally, in \cref{section:proof_lower_bound} we give a proof of \cref{theorem:lower_bound} by generalizing the construction to all heights.
\begin{remark}
    As in \cite{graham2001bipartite}, we quantify the respective sizes of the sets and other quantities with concrete numerical values.
    For example, ``large'' subsets of the vertex classes $A$ and $B$ of $D_0$ will mean $A' \subseteq A$ and $B' \subseteq B$ of sizes at least $0.98|A|$ and $0.98|B|$, respectively.
    However, we want to stress out that the actual numerical values are not of much importance; what matters is that the dependencies between them work out correctly.  
\end{remark}

\subsection{Finding the guest graph for the bipartite case}\label{section:guest_graph}
We begin by constructing the bipartite digraph $D_0$, which we will require to have bounded degree and satisfy certain pseudorandom properties. The following lemma, showing a graph with such properties can be constructed randomly, follows from a standard union bound argument, which we include for completeness. We use the notation $e_H(X,Y)$ to denote the number of pairs in $X \times Y$ that are edges of $H$.
\begin{lemma}\label{lemma:guest_graph}
    There exist constants $c_0 > c_1 > 1$ and $\Delta_0$ such that for each $\Delta \geq \Delta_0$ and $n \geq (c_0)^{2\Delta}$ there exists a bipartite graph $H$ with vertex classes $X$ and $Y$ of size $n$ and maximum degree at most $\Delta$ such that the following hold, where $k=(c_0)^{\Delta}$.
    \begin{enumerate}
        \item For all partitions $X = X_1 \cup \dots \cup X_k \cup D_X$ and $Y = Y_1 \cup \dots \cup Y_k \cup D_Y$ with $|X_i|, |Y_i| \leq (c_1 / c_0)^\Delta n$ and $|D_X|, |D_Y| \leq 0.02n$, we have
        \[
            \sum_{i \neq j: e_H(X_i, Y_i) > 0} |X_i||Y_j| > 0.55(0.98n)^2.
        \]
        \item For all $X' \subseteq X$ and $Y' \subseteq Y$ such that $|X'|, |Y'| \geq 0.01n$, we have $e_H(X', Y') > 0$.
    \end{enumerate}
\end{lemma}

\begin{proof}
    We take any $c_0,c_1$ satisfying $1 < c_1^2 < c_0 < (5/4)^{1/202}$ and choose $\Delta_0$ so that $(c_1^2/c_0)^{\Delta_0} < 0.1$, $((0.8)^{1/101}c_0^2)^{\Delta_0} < 1/16$ and $(1 - 10^4)^{\Delta_0/101} < 1/8$. Note that we can choose such a $\Delta_0$ since all three of these inequalities are satisfied for sufficiently large $\Delta_0$. Let moreover $\Delta \geq \Delta_0$, $d = \Delta/101$, and $m = 1.01n$.

    To obtain our graph $H$, we will first draw a bipartite graph $G$ uniformly at random from the set of all bipartite graphs with $dm$ edges and with vertex classes $V'$ and $V''$ of size $m$ each. 
    Then we will remove the $n/100$ largest degree vertices on each side to obtain $H$.
    Since the number of vertices of degree larger than $\Delta$ in $D$ is at most $\frac{dm}{\Delta + 1} < \frac{m}{101}=\frac n{100}$, the maximum degree of $H$ is at most $\Delta$ with probability $1$.
    It thus suffices to show that $H$ will also satisfy the other two properties with positive probability.

    For the first one, let us bound the probability that there exist partitions $V' = V_1' \cup \dots \cup V_k' \cup D_X \cup D'$ and $V'' = V_1'' \cup \dots \cup V_k'' \cup D_Y \cup D''$ with $|D'| = |D''| = n/100$,  $|D_X|, |D_Y| \leq n/50$, and $|V_i'|, |V_i''| \leq (c_1/c_0)^\Delta n$ for all $i \in [k]$, such that
    \[
        \sum_{i \neq j: e_H(V_i',V_j'') > 0} |V_i'||V_j''| \leq 0.55(0.98n)^2.
    \]
    To do that, notice that since 
    \[
        \sum_{i=1}^k |V_i'||V_i''| \leq k \left(\left(\frac{c_1}{c_0}\right)^{\Delta}n\right)^2 = \left(\frac{c_1^2}{c_0}\right)^\Delta n^2 \leq \left(\frac{c_1}  {c_0}\right)^{\Delta_0} n^2 < 0.1n^2,
    \]
    such a partition must satisfy
    \[
        \sum_{i \neq j: e_G(V_i', V_j'') = 0} |V_i'||V_j''| \geq (0.98n)^2 - 0.1n^2 - 0.55(0.98n)^2 \geq 0.2m^2.
    \]
    Therefore, by the union bound, the probability that such a partition exists is at most
    \[
        (k+2)^{2m} 2^{k^2} \frac{\binom{0.8m^2}{dm}}{\binom{m^2}{dm}} < (2k)^{2m} 2^{k^2} (0.8)^{dm} < 8^m ((0.8)^{1/101}c_0^2)^{\Delta_0 m} < \frac 12,
    \]
    where $(k+2)^{2m}$ is a bound on the number of partitions, $2^{k^2}$ bounds the number of possible choices of pairs $(V_i', V_j'')$ with no edges in between them and ${\binom{0.8m^2}{dm}}/{\binom{m^2}{dm}}$ is a bound on the probability that indeed no edges fall between them.

    Similarly, the probability that there exist $X' \subseteq X$ and $Y' \subseteq Y$ of sizes at least $0.01n$ each such that $e_H(X', Y') = 0$ is at most
    \[
        2^{2m} \frac{\binom{(1 - 10^4)m^2}{dm}}{\binom{m^2}{dm}} < 2^{2m} (1 - 10^4)^{dm} \leq 2^{2m} (1 - 10^4)^{\Delta_0 m / 101} < \frac 12,
    \]
    where $2^{2m}$ bounds the number of choices of $X'$ and $Y'$ and the fraction bounds the probability that there are no edges between them.
    Thus, by the union bound, there exists a choice of $G$ such that both properties are satisfied, implying the existence of the desired $H$. 
\end{proof}

\subsection{Finding the host tournament for the bipartite case}\label{section:host_graph}
Our next lemma is of a similar flavor to \cref{lemma:guest_graph}, showing that a random object typically satisfies a certain pseudorandom property. In this case, we show that a random tournament typically has many directed edges from any large set to any other large set. We actually prove something slightly more general, which says the same not for sets, but for ``weighted sets'', that is, for functions valued in $[0,1]$. Here, and in the rest of the proof, all logarithms are to base $e$.
\begin{lemma}\label{lemma:host_graph}
    Let $k \geq 2$ and $x \geq (10^8 \log k) /2$. There exists a tournament $R$ with vertex set $[k]$ such that for all pairs of weight functions $f, g: [k] \to [0,1]$ with $f+g \leq 1$ and $\sum_{i=1}^k (f(i) + g(i)) = 2x$, we have
    \[
        W \coloneqq \sum_{ij \in E(R)} f(i)g(j) \leq 0.51x^2.
    \]
\end{lemma}
\begin{proof}
    Note that $2x = \sum_{i=1}^k (f(i)+g(i)) \leq k$, hence the statement is vacuous if $2x>k$, as in this case there exist no such functions $f,g$. Thus, we assume henceforth that $2x \leq k$, and in particular that $k > 10^8 \log 2$. We let $R$ be a uniformly random tournament with vertex set $[k]$.

    We first claim that we can assume that there exist $i_0 \neq j_0$ such that for all $i \neq i_0$ and all $j \neq j_0$ we have
    \[
        f(i), g(j) \in \{0,1\}.
    \]
    Indeed, for any fixed outcome of $R$, suppose that $f$ and $g$ maximize $W$ and that there exists an $i$ such that $0 < f(i), g(i) < 1$.
    Now consider the sums $W_f(i) \coloneqq \sum_{j:ij \in E(R)} g(j)$ and $W_g(i) \coloneqq \sum_{j: ji \in E(R)} f(j)$.
    If $W_f(i) \geq W_g(i)$ then define new functions $f'$ and $g'$ such that there are equal to $f$ and $g$ except that $f'(i) = f(i) + g(i)$ and $g'(i) = 0$.
    Otherwise, we set $f'(i) = 0$ and $g'(i) = f(i) + g(i)$.
    In either case, we have $W' \geq W$ for the corresponding quantity $W'$.

    Thus, we can assume that $\min \{ f(i), g(i)  \} = 0$ for all $i$. 
    Now suppose that there exist $i \neq j$ such that $0 < g(i), g(j) < 1$.
    Again in case $W_g(i) \geq W_g(j)$, we let $\varepsilon_{ij} = \min \{ g(j), 1 - g(i) \}$ and let $g'(i) = g(i) + \varepsilon_{ij}$ and $g'(j) = g(j) - \varepsilon_{ij}$.
    Otherwise, we do the same with $i$ and $j$ swapped and in both cases we get $W' \geq W$.
    By the same argument, we can also assume that for at most one $i_0$ we have $0 < f(i_0) < 1$.

    Now, assuming $f$ and $g$ satisfy the property above, define $T = \{i: f(i) = 1\}$ and $S = \{ j: g(j) = 1\}$ and let $t = |T|, s = |S|$.
    By our assumptions, we have that $t+s \leq 2x < t + s + 2$ and $2x \leq k$.
    Additionally, we have
    \[
    W = \sum_{ij \in E(R)} f(i)g(j) = e_R(T, S) + g(j_0)W_g(j_0) + f(i_0)W_f(i_0) \leq e_R(T, S) + 2x.
    \]
    In particular, if $W > 0.51x^2$ we find that $$e_R(T,S) > 0.51x^2 - 2x \geq 0.501x^2 \geq 0.501\frac{(s+t)^2}{4}.$$
    We now claim that with positive probability (over the randomness in $R$), there exist no sets $S,T \subseteq V(R)$ satisfying this inequality.
    Suppose first that $t \leq s$.
    Note that we must have $t > s/7$ since otherwise $e_R(T, S)  \leq ts < 0.5(t+s)^2 / 4$.
    Similarly, we must have $s > s_0 \coloneqq 2 \cdot 10^7 \log k$.
    For any fixed disjoint $T,S$ we have that $e_R(T, S) \sim \operatorname{Bin}(ts, \frac 12)$ and thus by Chernoff's inequality, and using $(t+s)^2/4 \geq ts \geq s^2/7$, we get
    \[
        \Pr[e_R(T, S) > 0.501 (t+s)^2 /4] \leq \Pr[e_R(T, S) > 0.501ts] < e^{-10^{-7}s^2}.
    \]
    Moreover, for $s>s_0$, we have
    \[
        e^{-10^{-7}s^2} < e^{-10^{-7}s\cdot s_0}=k^{-2s}.
    \]
    Therefore, the probability that such $T$ and $S$ exist is at most
    \[
        \sum_{s=s_0}^{k} \sum_{t=s/7}^{s} \binom{k}{s}\binom{k}{t}e^{-10^{-7}s^2} \leq \sum_{s=s_0}^{k} \sum_{t=s/7}^{s} \left(\frac{ek}s\right)^{2s}k^{-2s} \leq k \sum_{s=s_0}^{k}\left(\frac{e^2}{s^2}\right)^s < k^2\left(\frac{e^2}{s_0^2}\right)^{s_0} < \frac 12,
    \]
    where in the first inequality we use that since $t \leq s$, we have that $\binom{k}{s} \binom{k}{t} \leq \binom{k}{s}^2 < (ek/s)^{2s}$. By interchanging the roles of $s$ and $t$, we obtain the same bound in case $t \geq s$. Thus,  we find that $R$ satisfies the desired property with positive probability.
\end{proof}

\subsection{The bipartite case}\label{section:bipartite_case}
We are now ready to prove \cref{theorem:lower_bound} in the case $h=2$, that is, when $D$ is bipartite. As discussed above, this step is actually the heart of the proof, as the proof for arbitrary $h$ will essentially be a reduction to this case.
\begin{lemma}\label{lemma:bipartite_lower_bound}
    There exist constants $c' > 1$ and $\Delta_0$ such that for all $\Delta \geq \Delta_0$ and $n \geq \Delta $ the following holds. 
    There exists a bipartite digraph $D_0 = (A \cup B, E)$ with $|A| = |B| \leq n$ and maximum degree $\Delta$  such that all its edges are directed from $A$ to $B$, as well as a tournament $R$ on $(c')^\Delta n$ vertices such that for any $A' \subseteq A$ and $B' \subseteq B$, the following hold.
    \begin{enumerate}
        \item If $|A'|, |B'| \geq 0.98|A|$ then there is no copy of $D_0[A' \cup B']$ in $R$, {and}
        \item if $|A'|, |B'| \geq 0.01|A|$ then $e_{D_0}(A', B') > 0$.
    \end{enumerate}
\end{lemma}
\begin{proof}
    Let $c_0, c_1$ and $\Delta_0$ be the constants from \cref{lemma:guest_graph}. By potentially increasing $\Delta_0$ further, we may also assume that $(c_0 / c_1)^{\Delta} > 10^9 \Delta \log c_0$ for all $\Delta \geq \Delta_0$.
    We let $1 < c' = \min\{ c_1, 2^{0.3}/c_0\}$.

    If $\Delta \geq \Delta_0$ and $\Delta  \leq n < \frac{1}{0.98}2c_0^\Delta$, let $D_0$ be the oriented complete bipartite graph $\vv{K}_{\Delta ,  \Delta }$ with vertex classes $A$ and $B$ and where all the edges are oriented from $A$ to $B$.
    Clearly, for any non-empty $A' \subseteq A$ and $B \subseteq B'$ we have $e_D(A', B') > 0$.
    To show that the first property also holds, note that if $|A'|, |B'| > 0.98 \Delta $ then $D_0[A' \cup B']$ has the complete bipartite graph $\vv{K}_{\lfloor0.98\Delta\rfloor, \lfloor0.98\Delta\rfloor}$ as a subgraph.
    Therefore, since the probability that a uniformly random tournament $R$ on $N = 2^{\lfloor0.98\Delta\rfloor/2} \geq (c')^\Delta n$ vertices contains a copy of $\vv{K}_{\lfloor0.49\Delta\rfloor, \lfloor0.49\Delta\rfloor}$ is at most $\binom{N}{\lfloor0.98\Delta\rfloor}\binom{N}{\lfloor0.98\Delta\rfloor} 2^{-\lfloor0.98\Delta\rfloor^2} < N^{2\lfloor0.98\Delta\rfloor} 2^{-\lfloor0.98\Delta\rfloor^2} \leq 1$, there is a choice of $R$ such that the conditions of the lemma are satisfied.

    Otherwise, we have $0.98n \geq 2c_0^\Delta$.
    Let then $D_0$ be a digraph obtained by taking the graph $H$ from Lemma \ref{lemma:guest_graph} and orienting every edge from $A$ to $B$, and let $R'$ be the tournament from Lemma \ref{lemma:host_graph}.
    We obtain a tournament $R$ by taking the blow-up of $R'$ in the following way.
    Let $N = c_1^\Delta n \geq (c')^\Delta n$, $k = c_0^\Delta$ and partition $[N] = U_1 \cup \dots \cup U_k$ such that $|U_i| = N/k$. 
    Let $R$ be an arbitrary tournament on the vertex set $[N]$ such that for all $ij \in E(R')$ we have $U_i \times U_j \subseteq E(R)$. That is, the edges between distinct $U_i,U_j$ form a blown-up copy of $R'$, and the edges inside any $U_i$ are oriented arbitrarily.
    
    $D_0$ satisfies the second condition of the lemma by \cref{lemma:guest_graph}.
    Suppose now for contradiction that for some $A' \subseteq A$ and $ B' \subseteq B$ with $|A'|, |B'| \geq 0.98n$ there is a copy of $D_0[A' \cup B']$ in $R$ and let $X$ and $Y$ be its two vertex classes.
    By Lemma \ref{lemma:guest_graph} we have that for $X_i \coloneqq X \cap U_i$ and $Y_i \coloneqq Y \cap U_i$ it holds
    \[
        \sum_{ij \in E(R')} |X_i||Y_j| \geq \sum_{i \neq j: e_{D}(X_i, X_j) > 0} |X_i||Y_j| > 0.55(0.98n)^2.
    \]
    For $i \in [k]$, let now $f(i) = \frac{|X_i|k}{N}$ and $g(i) = \frac{|Y_i|k}{N}$. We have $0 \leq f+g \leq 1$ and
    \[
        2x \coloneqq \sum_i (f(i) + g(i)) = \frac k N(\ab X + \ab Y)\geq  2 \cdot 0.98 \frac{nk}{N} = 1.96\left( \frac{c_0}{c_1} \right)^{\Delta} \geq 10^9 \Delta \log c_0 > 10^8 \log k.
    \]
    Therefore, by \cref{lemma:host_graph},
    \[
        \sum_{ij \in E(R')} |X_i||Y_j| = \frac{N^2}{k^2} \sum_{ij \in E(R)}f(i)g(j) < \frac{N^2}{k^2} 0.51x^2 \leq 0.51 (0.98n)^2,
    \]
    a contradiction. This shows that there is no copy of $D_0[A' \cup B']$ in $R$ for any such $A'$ and $B'$. 
\end{proof}

\subsection{Proof of Theorem \ref{theorem:lower_bound}}\label{section:proof_lower_bound}
With the ingredients above, we are ready to prove \cref{theorem:lower_bound}. 
\begin{proof}[Proof of Theorem \ref{theorem:lower_bound}]
    Let $c'>1$ and $\Delta'_0$ be the constants from Theorem \ref{lemma:bipartite_lower_bound} and set $\Delta_0$ to be a constant such that $\Delta_0 \geq \Delta_0'$ and $(c')^{\Delta_0/2} > 4$.
    Moreover, let $c>1$ be a constant such that $c^{\Delta_0} = (c')^{\Delta_0/2}/4$.
    Let $\Delta \geq 2\Delta_0$ and $n \geq \Delta $.
    Let now $D_0 = (A \cup B, E)$ and $R$ be respectively the bipartite digraph and tournament from Lemma \ref{lemma:bipartite_lower_bound}, applied with the parameters $n$ and $\Delta/2$.

    If $h=2$, then we take $D = D_0$ and since $R$ doesn't contain a copy of $D_0$ we get $\vv{r}(D) \geq (c')^{\Delta/2} n \geq c^\Delta n$.
    Otherwise, we define a graded digraph $D$ on $nh$ vertices and with a graded partition $V_1 \cup \dots \cup V_h$, where $\ab{V_i}=n$ for all $i$, by declaring that
    for all $i \in [h-1]$, the induced subgraph $D[V_i \cup V_{i+1}]$ is a copy of $D_0$ such that $V_i$ plays the role of $A$ and $V_{i+1}$ plays the role of $B$. Note that the maximum degree in $D$ is at most $\Delta$, since the maximum in-degree and maximum out-degree are both at most $\Delta/2$.

    Now let $H = \lceil\frac{h}{2}\rceil -1  > 1$ and define a tournament $T$ on $(c')^{\Delta/2} Hn$ vertices with a vertex partition $V(T) = A_1 \cup \dots \cup A_H$, where
    $|A_i| = (c')^{\Delta/2} n$ for each $i \in [H]$, as follows. For each $i$, we let
    $T[A_i]$ be a copy of $R$, and for all $1 \leq i<j\leq H$, we direct all edges from $A_i$ to $A_j$. 
    We claim that there is no copy of $D$ in $T$.
    
    Indeed, suppose for contradiction that $\phi$ is an embedding of $D$ into $T$.
    For $i \in [h]$ and $A' \subseteq V(T)$, let $f_i(A') = \frac{|\phi(V_i) \cap A'|}{|V_i|}$ be the fraction of vertices of $V_i$ embedded into $A'$. Additionally, let $U_i = \bigcup_{j=i}^H A_j$, and let $j_i$ be the largest index $1 \leq j \leq H$ such that $f_i(U_j) \geq 0.99$. Note that this is well-defined since $U_1=V(T)$, and hence $f_i(U_1)=1$. 

    The key observation is that if $f_i(U_j)\geq 0.01$, then $f_{i+1}(U_j) \geq 0.99$. Indeed, if this is not the case, then there exist $X_i \subseteq V_i, X_{i+1} \subseteq V_{i+1}$ with $\ab{X_i},\ab{X_{i+1}}\geq 0.01 n$, with the properties that $\phi(X_i) \subseteq U_j$ and $\phi(X_{i+1}) \subseteq V(T) \setminus U_j$. But all edges in $T$ are directed from $V(T) \setminus U_j$ to $U_j$, hence the second condition in \cref{lemma:bipartite_lower_bound} implies that if such $X_i,X_{i+1}$ exist, then $\phi$ is not a valid embedding.

    In particular, applying this observation with $j=j_i$, we conclude that $f_{i+1}(U_{j_i}) \geq 0.99$. This implies that $j_{i+1}\geq j_i$ for all $i \in [h-1]$, that is, that the indices $j_i$ are monotonically non-decreasing. 
    We now claim that for each $i \in [h-2]$, we have that $j_{i+2}>j_i$.

    Indeed, if $j_{i+1}>j_i$, then we are done by the monotonicity property $j_{i+2}\geq j_{i+1}$. Hence we may assume that $j_i=j_{i+1}$, which in particular implies that $f_i(U_{j_i+1}) < 0.01$ by the key observation. If $f_{i+1}(U_{j_i+1})\geq 0.01$, then we are again done by the key observation. Therefore, we may assume that $f_{i+1}(U_{j_i+1})<0.01$. Together with the fact that $j_{i+1}=j_i$, we conclude that $f_i(A_{j_i}), f_{i+1}(A_{j_i})\geq 0.98$. 

    In other words, there exist $X_i \subseteq V_i, X_{i+1} \subseteq V_{i+1}$ with $\ab{X_i},\ab{X_{i+1}} \geq 0.98n$ such that $\phi(X_i), \phi(X_{i+1}) \subseteq A_{j_i}$. But this is a contradiction to \cref{lemma:bipartite_lower_bound}, since $T[A_{j_i}]$ is a copy of $R'$. We conclude that, as claimed, $j_{i+2}> j_i$ for all $i \in [h-2]$.

    Since $j_1 \geq 1$ and $j_{i+2} \geq j_i + 1$ for all $i$, we find that $j_i \geq i/2$ for all $i$. In particular, $j_h \geq h/2 > H$. But this is a contradiction as there are only $H$ parts in $T$, implying that there is no copy of $D$ in $T$.
    Therefore, 
    \[
        \vv{r}(D) >  (c')^{\Delta/2} Hn \geq \frac 14 (c')^{\Delta/2} hn \geq c^\Delta hn.\qedhere
    \]
\end{proof}

\section{Concluding remarks}\label{sec:conclusion}
While \cref{theorem:easy_upper_bound} is roughly best possible in general, it is reasonable to expect that one could improve it in certain cases. In particular, for the oriented hypercube $\vv{Q_d}$, we expect that the bound in \cref{thm:hypercube} could be significantly improved.
In fact, our techniques are already sufficient show that the induced subgraph of $\vv{Q_d}$ obtained by taking the vertices with at most $d/2$ non-zero coordinates has oriented Ramsey number at most $2^{3d+o(d)}$. As this digraph consists of simply the first half of the graded partition of $\vv{Q_d}$, this suggests to us that the bound in \cref{thm:hypercube} is not particularly close to best possible.
Concretely, we make the following conjecture, which is a directed analogue of the
Burr--Erd\H{o}s conjecture \cite{burr_magnitude_1975} that $r(Q_d) = O(2^d)$.
\begin{conjecture}
    There is an absolute constant $C>0$ such that $\vv r(\vv{Q_d}) \leq C2^d$ for all $d \geq 1$.
\end{conjecture}

We also recall that Fox, He, and Wigderson \cite{fox2021ramsey} proved that any bounded-height, bounded-degree acyclic digraph has linear oriented Ramsey number, and \cref{theorem:easy_upper_bound} implies that the same is true for graded digraphs even of unbounded height. It would be very interesting to extend this result, and identify other classes of unbounded-height digraphs where an analogue of the Burr--Erd\H os conjecture holds.

Finally, we reiterate the main question left open from \cite{fox2021ramsey}, which we consider to be of central importance in the study of oriented Ramsey numbers.
\begin{question}
    Given $\Delta \geq 1$, does there exist some $C>0$ such that every $n$-vertex acyclic digraph $D$ with maximum degree $\Delta$ satisfies $\vv r(D) \leq n^C$?
\end{question}

\paragraph{Acknowledgments:} We are grateful to Domagoj Brada\v c and Benny Sudakov for helpful discussions, and for suggesting a simplification to the proof of \cref{lemma:dependent_random_choice}. We would also like to thank Xioayu He for constructive comments on an earlier draft of this paper.


\begin{thebibliography}{10}
\providecommand{\url}[1]{\texttt{#1}}
\providecommand{\urlprefix}{URL }
\providecommand{\eprint}[2][]{\url{#2}}

\bibitem{beck1983upper}
J.~Beck, An upper bound for diagonal {R}amsey numbers, \emph{Studia Sci. Math. Hungar.} \textbf{18} (1983), 401--406.

\bibitem{benford2022few}
A.~Benford and R.~Montgomery, Trees with few leaves in tournaments, \emph{J. Combin. Theory Ser. B} \textbf{155} (2022), 141--170.

\bibitem{benford2022many}
A.~Benford and R.~Montgomery, Trees with many leaves in tournaments, 2022. Preprint available at arXiv:2207.06384.

\bibitem{bradac2023turan}
D.~Brada\v{c}, O.~Janzer, B.~Sudakov, and I.~Tomon, The {T}ur\'{a}n number of the grid, \emph{Bull. Lond. Math. Soc.} \textbf{55} (2023), 194--204.

\bibitem{bucic2019directed}
M.~Buci\'{c}, S.~Letzter, and B.~Sudakov, Directed {R}amsey number for trees, \emph{J. Combin. Theory Ser. B} \textbf{137} (2019), 145--177.

\bibitem{burr_magnitude_1975}
S.~A. Burr and P.~Erd\H{o}s, On the magnitude of generalized {R}amsey numbers for graphs, in \emph{Infinite and finite sets ({C}olloq., {K}eszthely, 1973; dedicated to {P}. {E}rd\H{o}s on his 60th birthday), {V}ols. {I}, {II}, {III}}, \emph{Colloq. Math. Soc. J\'{a}nos Bolyai}, vol. Vol. 10, North-Holland, Amsterdam-London, 1975,  215--240.

\bibitem{chen1993graphs}
G.~Chen and R.~H. Schelp, Graphs with linearly bounded {R}amsey numbers, \emph{J. Combin. Theory Ser. B} \textbf{57} (1993), 138--149.

\bibitem{chvatal1983ramsey}
C.~Chvat\'{a}l, V.~R\"{o}dl, E.~Szemer\'{e}di, and W.~T. Trotter, Jr., The {R}amsey number of a graph with bounded maximum degree, \emph{J. Combin. Theory Ser. B} \textbf{34} (1983), 239--243.

\bibitem{clemens2021size}
D.~Clemens, M.~Miralaei, D.~Reding, M.~Schacht, and A.~Taraz, On the size-{R}amsey number of grid graphs, \emph{Combin. Probab. Comput.} \textbf{30} (2021), 670--685.

\bibitem{conlon2009hypergraph}
D.~Conlon, Hypergraph packing and sparse bipartite {R}amsey numbers, \emph{Combin. Probab. Comput.} \textbf{18} (2009), 913--923.

\bibitem{conlon2013ramsey}
D.~Conlon, The {R}amsey number of dense graphs, \emph{Bull. Lond. Math. Soc.} \textbf{45} (2013), 483--496.

\bibitem{conlon2012two}
D.~Conlon, J.~Fox, and B.~Sudakov, On two problems in graph {R}amsey theory, \emph{Combinatorica} \textbf{32} (2012), 513--535.

\bibitem{conlon2016short}
D.~Conlon, J.~Fox, and B.~Sudakov, Short proofs of some extremal results {II}, \emph{J. Combin. Theory Ser. B} \textbf{121} (2016), 173--196.

\bibitem{conlon2023size}
D.~Conlon, R.~Nenadov, and M.~Truji\'{c}, On the size-{R}amsey number of grids, \emph{Combin. Probab. Comput.} \textbf{32} (2023), 874--880.

\bibitem{draganic2021powers}
N.~Dragani\'{c}, F.~Dross, J.~Fox, A.~Gir\~{a}o, F.~Havet, D.~Kor\'{a}ndi, W.~Lochet, D.~Munh\'{a}~Correia, A.~Scott, and B.~Sudakov, Powers of paths in tournaments, \emph{Combin. Probab. Comput.} \textbf{30} (2021), 894--898.

\bibitem{dross2021unavoidability}
F.~Dross and F.~Havet, On the unavoidability of oriented trees, \emph{Journal of Combinatorial Theory, Series B} \textbf{151} (2021), 83--110.

\bibitem{el2004trees}
A.~El~Sahili, Trees in tournaments, \emph{J. Combin. Theory Ser. B} \textbf{92} (2004), 183--187.

\bibitem{erdos1975partition}
P.~Erd\H{o}s and R.~L. Graham, On partition theorems for finite graphs, in \emph{Infinite and finite sets ({C}olloq., {K}eszthely, 1973; dedicated to {P}. {E}rd\H{o}s on his 60th birthday), {V}ols. {I}, {II}, {III}}, \emph{Colloq. Math. Soc. J\'{a}nos Bolyai}, vol. Vol. 10, North-Holland, Amsterdam-London, 1975,  515--527.

\bibitem{lovasz_lemma}
P.~Erd\H{o}s and L.~Lov\'{a}sz, Problems and results on {$3$}-chromatic hypergraphs and some related questions, in \emph{Infinite and finite sets ({C}olloq., {K}eszthely, 1973; dedicated to {P}. {E}rd\H{o}s on his 60th birthday), {V}ols. {I}, {II}, {III}}, \emph{Colloq. Math. Soc. J\'{a}nos Bolyai}, vol. Vol. 10, North-Holland, Amsterdam-London, 1975,  609--627.

\bibitem{erdos1964representation}
P.~Erd\H{o}s and L.~Moser, On the representation of directed graphs as unions of orderings, \emph{Magyar Tud. Akad. Mat. Kutat\'{o} Int. K\"{o}zl.} \textbf{9} (1964), 125--132.

\bibitem{fox2021ramsey}
J.~Fox, X.~He, and Y.~Wigderson, Ramsey numbers of sparse digraphs, \emph{Israel J. Math.}  (2024), to appear. Preprint available at arXiv:2105.02383.

\bibitem{fox2009density}
J.~Fox and B.~Sudakov, Density theorems for bipartite graphs and related {R}amsey-type results, \emph{Combinatorica} \textbf{29} (2009), 153--196.

\bibitem{fox2009two}
J.~Fox and B.~Sudakov, Two remarks on the {Burr-Erd\H{o}s} conjecture, \emph{European J. Combin.} \textbf{30} (2009), 1630--1645.

\bibitem{fox2011dependent}
J.~Fox and B.~Sudakov, Dependent random choice, \emph{Random Structures Algorithms} \textbf{38} (2011), 68--99.

\bibitem{furedi2013uniform}
Z.~F\"{u}redi and M.~Ruszink\'{o}, Uniform hypergraphs containing no grids, \emph{Adv. Math.} \textbf{240} (2013), 302--324.

\bibitem{gao2023extremal}
J.~Gao, O.~Janzer, H.~Liu, and Z.~Xu, Extremal number of graphs from geometric shapes, \emph{Israel J. Math.}  (2024), to appear. Preprint available at arXiv:2303.13380.

\bibitem{gishboliner2022constructing}
L.~Gishboliner and A.~Shapira, Constructing dense grid-free linear 3-graphs, \emph{Proc. Amer. Math. Soc.} \textbf{150} (2022), 69--74.

\bibitem{graham2000graphs}
R.~L. Graham, V.~R\"{o}dl, and A.~Ruci\'{n}ski, On graphs with linear {R}amsey numbers, \emph{J. Graph Theory} \textbf{35} (2000), 176--192.

\bibitem{graham2001bipartite}
R.~L. Graham, V.~R\"{o}dl, and A.~Ruci\'{n}ski, On bipartite graphs with linear {R}amsey numbers, \emph{Combinatorica} \textbf{21} (2001), 199--209.

\bibitem{haggkvist1991trees}
R.~H\"{a}ggkvist and A.~Thomason, Trees in tournaments, \emph{Combinatorica} \textbf{11} (1991), 123--130.

\bibitem{havet2002trees}
F.~Havet, Trees in tournaments, \emph{Discrete Math.} \textbf{243} (2002), 121--134.

\bibitem{havet2000median}
F.~Havet and S.~Thomass\'{e}, Median orders of tournaments: a tool for the second neighborhood problem and {S}umner's conjecture, \emph{J. Graph Theory} \textbf{35} (2000), 244--256.

\bibitem{havet2000oriented}
F.~Havet and S.~Thomass\'{e}, Oriented {H}amiltonian paths in tournaments: a proof of {R}osenfeld's conjecture, \emph{J. Combin. Theory Ser. B} \textbf{78} (2000), 243--273.

\bibitem{kim2016two}
J.~H. Kim, C.~Lee, and J.~Lee, Two approaches to {S}idorenko's conjecture, \emph{Trans. Amer. Math. Soc.} \textbf{368} (2016), 5057--5074.

\bibitem{kostochka2003ramsey}
A.~Kostochka and B.~Sudakov, On {R}amsey numbers of sparse graphs, \emph{Combin. Probab. Comput.} \textbf{12} (2003), 627--641.

\bibitem{kostochka2001graphs}
A.~V. Kostochka and V.~R\"{o}dl, On graphs with small {R}amsey numbers, \emph{J. Graph Theory} \textbf{37} (2001), 198--204.

\bibitem{kuhn2011approximate}
D.~K\"{u}hn, R.~Mycroft, and D.~Osthus, An approximate version of {S}umner's universal tournament conjecture, \emph{J. Combin. Theory Ser. B} \textbf{101} (2011), 415--447.

\bibitem{kuhn2011proof}
D.~K\"{u}hn, R.~Mycroft, and D.~Osthus, A proof of {S}umner's universal tournament conjecture for large tournaments, \emph{Proc. Lond. Math. Soc. (3)} \textbf{102} (2011), 731--766.

\bibitem{lee2017ramsey}
C.~Lee, Ramsey numbers of degenerate graphs, \emph{Ann. of Math. (2)} \textbf{185} (2017), 791--829.

\bibitem{mota2015ramsey}
G.~O. Mota, G.~N. S\'{a}rk\"{o}zy, M.~Schacht, and A.~Taraz, Ramsey numbers for bipartite graphs with small bandwidth, \emph{European J. Combin.} \textbf{48} (2015), 165--176.

\bibitem{shi2001cube}
L.~Shi, Cube {R}amsey numbers are polynomial, \emph{Random Structures Algorithms} \textbf{19} (2001), 99--101.

\bibitem{shi2007tail}
L.~Shi, The tail is cut for {R}amsey numbers of cubes, \emph{Discrete Math.} \textbf{307} (2007), 290--292.

\bibitem{stearns1959voting}
R.~Stearns, The voting problem, \emph{Amer. Math. Monthly} \textbf{66} (1959), 761--763.

\bibitem{sudakov2011conjecture}
B.~Sudakov, A conjecture of {Erd\H{o}s} on graph {R}amsey numbers, \emph{Adv. Math.} \textbf{227} (2011), 601--609.

\bibitem{thomason1986paths}
A.~Thomason, Paths and cycles in tournaments, \emph{Trans. Amer. Math. Soc.} \textbf{296} (1986), 167--180.

\bibitem{tikhomirov2024remark}
K.~Tikhomirov, A remark on the {R}amsey number of the hypercube, \emph{European J. Combin.} \textbf{120} (2024), Paper No. 103954.

\end{thebibliography}
\end{document}